\documentclass[12 pt]{article}
\usepackage[left=3.5cm,right=3cm,
    top=3cm,bottom=3.5cm,bindingoffset=0cm]{geometry}
\usepackage{amssymb,amsmath,amsfonts,amsthm}
\usepackage{times}
\usepackage{amsfonts}
\usepackage{amsthm}
\usepackage{amsmath}
\usepackage{mathrsfs}
\usepackage{epic}
\usepackage{times}

\usepackage{graphicx}
\usepackage{subcaption}
\usepackage{float}
 2

\newcommand{\nats}{\mathbb N}
\newcommand{\A}{\mathbb A}

\newcommand{\less}{\prec}

\newcommand{\Lyn}{\mathscr{L}_\omega}

\newtheorem{theorem}{Theorem}[section]
\newtheorem{corollary}[theorem]{Corollary}
\newtheorem{lemma}[theorem]{Lemma}
\newtheorem{proposition}[theorem]{Proposition}

\theoremstyle{definition}
\newtheorem{remark}{Remark}
\newtheorem{definition}[theorem]{Definition}

\newenvironment{dedication}
  {\vspace*{1cm} 
   \itshape  
   \raggedleft   
  }
  {\par % end the paragraph
  }

\begin{document}
\title{\bf $\omega$-Lyndon words}
\author{Micka\"el Postic\textsuperscript{1} and Luca Zamboni\textsuperscript{1}}
\date{}
\maketitle
\begin{center}
\textsuperscript{1}Universit\'{e} de Lyon, Universit\'{e} Lyon  1,  CNRS  UMR  5208,  Institut  Camille  Jordan,  43 boulevard du 11 novembre 1918, F69622 Villeurbanne Cedex, France\\
\texttt{\{postic,zamboni\}@math.univ-lyon1.fr}

\begin{dedication}
\textbf{In Memory of the Late Professor Aldo de Luca }
\end{dedication}
\end{center}

\vspace{0.5cm}

{\bf Abstract:} {Let $\A$ be a finite non-empty set and $\preceq $  a total order on   $\A^\nats$  verifying the following lexicographic like condition:  For each $n\in \nats$ and  $u, v\in \A^n,$ if $u^\omega \prec v^\omega$   then  $ux\prec vy$ for all $x, y \in \A^\nats.$ A word $x\in \A^\nats$ is called $\omega$-Lyndon if $x\prec y$ for each proper suffix $y$ of $x.$  A finite word $w\in \A^+$ is called $\omega$-Lyndon if $w^\omega \prec v^\omega$ for each proper suffix $v$ of $w.$ In this note we prove that every infinite word may be written uniquely as a non-increasing product of $\omega$-Lyndon words.}\\

%\vspace{.05 in}


\section{Introduction}

Given a finite non-empty set $\A,$ let $\A^+=\bigcup_{n\in \nats}\A^n$ denote the free semigroup generated by $\A$ consisting of  all finite words over $\A$ and let $\A^\nats=\{a_1a_2a_3\cdots \,|\, a_i\in \A\}$ be  the set of all right infinite words over $\A.$ We also let $\A^*=\A^+\cup \{\varepsilon\}$ where $\varepsilon$ denotes the empty word.  Let $\preceq $ be a total order on $\A^\nats$ verifying the following lexicographic like condition denoted (*):  For each $n\in \nats$ and $u, v\in \A^n$ if $u^\omega \less v^\omega$  then  $ux\less vy$ for all $x, y \in \A^\nats.$ 
A special case of this setting was previously considered in \cite{DRR}. The authors consider a sequence  $(<_n)_{n\in \nats}$ of total orders on $\A.$
 This induces a total lexicographic like order $\preceq$ on $\A^\nats$ defined by $x\preceq y$ if and only if either $x=y$ or if $x=uax'$ and $y=uby'$ for some $u\in \A^*,$ $a,b\in \A$ and $x',y' \in \A^\nats$ and $a<_{|u|+1}b.$ It is easily checked that the induced total order on $\A^\nats$ satisfies condition (*).  However the two settings are not equivalent. In fact, in the context of the generalised lexicographic order  in  \cite{DRR}, if for example  $(bb)^\omega \prec (ba)^\omega$ then it would mean that $b<_2a$ and hence $(ab)^\omega \prec (aa)^\omega.$ This is no longer true in the setting considered herein as one may have  $(bb)^\omega \prec (ba)^\omega$ and $(aa)^\omega \prec (ab)^\omega.$ 
In \cite{DRR} the authors introduce the notion of {\it generalised Lyndon words} with respect to the induced total order on $\A^\nats:$ a finite word $w\in \A^+$ is a generalised Lyndon word if $w^\omega \prec v^\omega$ for every proper suffix $v$ of $w.$ They then prove that every finite word $w\in \A^+$ may be written uniquely in the form $w=l_1\cdots l_k$ where each $l_i$ is a generalised Lyndon word and $l_1^\omega \succeq l_2^\omega \succeq \cdots \succeq l_k^\omega$ (see Theorem 16 in \cite{DRR}).  Analogously, given a total order on $\A^\nats$ verifying (*) we adopt the above definition to define the notion of  {\it $\omega$-Lyndon words}: a finite word $w$ is $\omega$-Lyndon if $w^\omega \prec v^\omega$ for every proper suffix $v$ of $w.$ We may also borrow from the usual definition of infinite Lyndon words to define a class of infinite $\omega$-Lyndon words: An infinite word $x\in \A^\nats$ is called $\omega$-Lyndon if $x\prec y$ for every proper suffix $y$ of $x.$  Alternatively,  to define an infinite $\omega$-Lyndon word, we could have adapted the  characterisation given in \cite{BC}: An infinite word $x$ is $\omega$-Lyndon if it is non-periodic and contains an infinite number of $\omega$-Lyndon prefixes.  It turns out that the two notions are equivalent (see Lemma~\ref{L1}). 

The main purpose of this note is to extend the results in \cite{SM} on factoring infinite words as a non-increasing product of Lyndon words. We prove that every infinite word may be written uniquely as a non-increasing product of $\omega$-Lyndon words.
This provides a positive answer to a question raised in \cite{DRR}. 

\section{Main results}

\noindent Let us fix once and for all a total order $\preceq $ on $\A^\nats$ verifying the lexicographic condition (*). Note that if $x,y\in \A^\nats$ and $x\preceq y$ then for each prefix $u$ of $x$ and $v$ of $y$ with $|u|=|v|$ one has $u^\omega \preceq v^\omega$ with equality if and only if $u=v.$ We also observe that if $u^\omega \preceq v^\omega,$ and neither $u$ nor $v$ is a prefix of the other, then $u^\omega \prec v^\omega.$ In particular, if $u^\omega \preceq v^\omega$ and $u$ and $v$ are primitive, then either $u=v$ or $u^\omega \prec v^\omega.$ 
We begin with the following  lemma which is analogous to Lemma 13 in \cite{DRR}. We omit the proof as it is identical to that of Lemma 13 in \cite{DRR}.  

\begin{lemma}\label {lem11} For each $u,v \in \A^+$ and $\star\in \{=,\less, \succ\}$ the following are equivalent :
\begin{enumerate}
\item $u^\omega \star\, v^\omega;$
\item $(uv)^\omega \star\, v^\omega;$
\item $u^\omega \star\, (vu)^\omega;$
\item $(uv)^\omega \star\, (vu)^\omega;$
\end{enumerate} 
\end{lemma} 

We remark that a slightly modified version of the above lemma also applies in case one of $u$ and $v$ is infinite and the other finite: For example if $u\in \A^+$ and $v\in \A^\nats$ then $u^\omega \star v$ if and only if $uv\star v.$

Given a total order  on $\A^\nats$ it is natural to consider the notion of a Lyndon word, namely an element $x\in \A^\nats$ which is smaller (relative to the prescribed total order) than each of its proper suffixes. It will be useful to extend this notion also to finite words, however given that the order is defined only on infinite words, we shall be required to pass to infinite words by associating to each finite word $w$ its (periodic) infinite counterpart $w^\omega.$ Following \cite{DRR}:  
 
\begin{definition}\rm An infinite word $x\in \A^\nats$ is called $\omega$-Lyndon if $x\less y$ for each proper suffix $y$ of $x.$  A word $w\in \A^+$ is called $\omega$-Lyndon if $w^\omega \prec v^\omega$ for each proper suffix $v$ of $w.$  We let $\Lyn$ denote the set of all $\omega$-Lyndon words (finite and infinite) relative to the total order $\preceq .$
\end{definition}

\begin{remark}\label{rem2}\rm  We note that  $ \A\subseteq \Lyn.$ If $w\in \A^+$ is $\omega$-Lyndon, then $w$ is primitive and similarly if $x\in \A^\nats$ is $\omega$-Lyndon, then $x$ is not periodic.  It follows from Lemma~\ref{lem11} that $w\in \A^+$ is $\omega$-Lyndon if and only if for all factorisations $w=uv$ with $u,v\in \A^+$ we have $u^\omega \less v^\omega$ (see Theorem 14 in \cite{DRR}).  This in turn implies that  if $w\in \Lyn,$ then for each prefix $u$ of $w$ and each factor $v$ of $w$ with $|u|=|v|,$ either $u=v$ or $u^\omega \prec v^\omega.$ In fact, suppose $u\neq v$ and let $z$ be a suffix of $w$ beginning in $v.$ Then  if $w\in \A^+$ we have that $w^\omega \prec z^\omega$ and hence $u^\omega \prec v^\omega.$ If $w\in \A^\nats,$ then $w\prec z$ and hence $u^\omega \prec v^\omega.$ \end{remark}

\begin{remark}\label{FILyn}\rm As is the case for usual Lyndon words, each finite $\omega$-Lyndon word $w=w_1w_2\cdots w_n$ of length $n\geq 2$ is a prefix of some infinite $\omega$-Lyndon word $w'.$ In fact, if $w$ is unbordered we may take $w'=ww_n^\omega.$ For let $s$ be a suffix of $w.$ Then 
as $w_1\cdots w_{|s|}\neq s$ we have $ (w_1\cdots w_{|s|})^\omega \prec s^\omega$ and hence $w' \prec sw_n^\omega .$  On the other hand, if $w$ is bordered, we may take $w'=wu$ where $u$ is the longest border of $w.$ In fact, let $s$ be a suffix of $w.$ Then if $s$ is not a border of $w$ then as before we deduce that  $w' \prec su^\omega .$ While if $s$ is a prefix of $w,$ and hence also a prefix of $u,$ let $k$ be the first position in which $w^\omega$ and $s^\omega$ first differ. If $k\leq |s|$, then $(w_1\cdots w_k)^\omega \prec (s_1\cdots s_k)^\omega$ and hence $w'\prec  su^\omega.$ If $|s|<k\leq |u|$ then  $(w_1\cdots w_k)^\omega \prec (su_1\cdots u_{k-|s|})^\omega$ and so again we have $w'\prec  su^\omega.$ Finally if $k>|u|$ then writing
$s=z^a$ and $u=z^b$ for some $z\in \A^+$ and $a,b\in \nats$ we again get $w'\prec su^\omega$ as required. 
A similar proof shows that each finite $\omega$-Lyndon word $w\notin \A$ is a prefix of a longer finite $\omega$-Lyndon word.
\end{remark}

\begin{remark}\label{usual}\rm In contrast to the previous remark, many fundamental properties of usual Lyndon words no longer hold for $\omega$-Lyndon words. First of all, every primitive finite word and every non-periodic infinite word is $\omega$-Lyndon relative to some total order $\preceq$ on $\A^\nats$ verifying (*). As a consequence, a finite $\omega$-Lyndon word need not be unbordered. Or an infinite $\omega$-Lyndon word $x$ may be a product of prefixes of $x.$ Or if $u,v\in \A^+$ are $\omega$-Lyndon and $u^\omega \less v^\omega$ it need not be the case that $uv$ is $\omega$-Lyndon.  For example, let $\A=\{a,b\}$ be ordered by $a<b.$ Consider the total order $\preceq$ on $\A^\nats$ defined as follows: For distinct $x,y\in \A^\nats$  consider the smallest $n$ with $x_n\neq y_n.$ Then set 
\[x\less y\Leftrightarrow \begin{cases}\,\,x_n<y_n\,\,\mbox{if $n$ is odd}\\\,\,y_n<x_n \,\,\mbox{if $n$ is even}\end{cases}\]
Then $u=abba\in \Lyn$ and $v=b\in \Lyn$ and $u^\omega \less v^\omega,$ yet $uv\notin \Lyn.$ This example also illustrates that $\omega$-Lyndon words may be bordered. 
\end{remark}

\medskip

\noindent The following proposition is essentially Theorem 16 in \cite{DRR}. As with Lemma~\ref{lem11}, we omit the proof as it is identical to that of Theorem 16 in \cite{DRR}.  

\begin{proposition}\label{prop1}Each $w\in \A^+$ admits a unique factorisation $w=l_1l_2\cdots l_k$ with $l_i\in \Lyn$ and $l_1^\omega \succeq l_2^\omega \succeq \cdots \succeq l_k^\omega.$
\end{proposition}

%\begin{proof} We begin by noting that if $w\in \A^+$ admits a factorisation of the form $w=l_1l_2\cdots l_k$ with $l_i\in \Lyn$ and $l_1^\omega \succeq l_2^\omega \succeq \cdots \succeq l_k^\omega$ then this factorisation is unique. In fact by iteration of Lemma~\ref{lem11} we have that $(l_i\cdots l_k)^\omega\succeq l_k^\omega $ for each $i\leq k.$ Moreover, as each $l_i\in \Lyn,$ it follows that $v^\omega \succeq l_k^\omega  $ for each suffix $v$ of $w$ and if $v$ is a suffix of $l_k$ then $v^\omega \succ l_k^\omega.$  Thus $l_k$ is the shortest suffix $u$ of $w$ such that $u^\omega \preceq v^\omega$ for each suffix $v$ of $w.$ It follows therefore by induction on $|w|$ that if $w$ admits a non-increasing $\omega$-Lyndon factorisation, then it is unique. To prove existence of the required factorisation we may also proceed by induction on $|w|.$ If $|w|=1,$ then $w\in \Lyn$ and hence we are done. Next suppose that $|w|\geq 2.$  Let $u$ be the shortest suffix of $w$ such that $u^\omega \preceq  v^\omega$ for each suffix $v$ of $w.$ Then if $u=u_1u_2$ one has that $u^\omega \prec u_2^\omega$  and hence $u\in \Lyn.$ If $w=u$ then we are done otherwise write $w=w'u$ and by induction we may write $w'=l_1l_2\cdots l_k$ with $l_i\in \Lyn$ and $l_1^\omega \succeq l_2^\omega \succeq \cdots \succeq l_k^\omega.$By definition of $u$ we have $(l_ku)^\omega \succeq u^\omega$ which by Lemma~\ref{lem11} is equivalent to $l_k^\omega \succeq u^\omega$ as required.\end{proof}

\begin{definition}For $x\in \A^\nats$ we say $x$ admits an infinite $\omega$-Lyndon factorisation  if $x=\prod_{i=1}^\infty l_i$ with each $l_i\in \Lyn\cap \A^+$ and $l_1^\omega \succeq l_2^\omega \succeq l_3^\omega \succeq \cdots.$ We say $x$ admits a finite  $\omega$-Lyndon factorisation if $x=l_1l_2\ldots l_k$ with  $l_i\in \Lyn$ and $l_1^\omega \succeq l_2^\omega \succeq \cdots \succeq  l_{k-1}^\omega\ \succ l_k.$\end{definition} 

\begin{remark}\rm Because of the fact that if $u$ and $v$ are Lyndon in the usual sense and $u<v$ then $uv$ is Lyndon, it follows that the factorisation of a finite word $w$ as a non-increasing product of Lyndon words is also the shortest factorisation of $w$ as a product of Lyndon words. This is no longer true in general for $\omega$-Lyndon words. For example, relative to the total order $\preceq$  defined in Remark~\ref{usual}, the word $w=ababab$ is the product of $ababa$ and $b,$ both of which are $\omega$-Lyndon, yet the $\omega$-Lyndon factorisation of $w$ has length three and is given by $w=(ab)(ab)(ab).$  
\end{remark}

\noindent The following lemma constitutes a generalisation of a characterisation of infinite Lyndon words given in \cite{BC}:
%Note that any finite $\omega$-Lyndon word is necessarily primitive and any infinite $\omega$-Lyndon word is non-periodic. 

\begin{lemma}\label{L1} Let $x\in \A^\nats .$ Then $x\notin \Lyn$ if and only if either $x=l^\omega$ for some $l\in \Lyn $ or only a finite number of prefixes of $x$ are members of $\Lyn.$
\end{lemma} 

\begin{proof} Assume $x\notin \Lyn$ and pick a  proper suffix $y$ of $x$ with $y\preceq x.$ If $y=x,$ then $x=u^\omega$ for some primitive word $u\in \A^+.$ 
As $u$ is primitive,  for each nontrivial factorisation $u=u_1u_2$ one has $u_2u_1\neq u$ and hence $(u_2u_1)^\omega \neq u^\omega.$  If $u\in \Lyn$ we are done. If $u\notin \Lyn,$  pick a factorisation $u=u_1u_2$ with $(u_2u_1)^\omega \prec u^\omega.$ Then by Remark~\ref{rem2}, if $v$ is a prefix of $x$ with $|v|\geq 2|u|$ then $v\notin \Lyn.$  If $y\prec x,$ pick a prefix $v$ of $y$ and a prefix $u$ of $x$ with $|u|=|v|$ and $v^\omega\prec u^\omega.$ Then again by Remark~\ref{rem2} any prefix of $x$ containing $v$ as a factor cannot belong to $\Lyn.$ 

For the converse, we note that if $x$ is periodic then $x\notin \Lyn.$ So assume $x$ is not periodic and only a finite number of prefixes of $x$ belong to $\Lyn.$
For $n\in \nats, $  let $x[n]$ denote the prefix of $x$ of length $n,$ and let $l(n)$ denote the length of a longest $\omega$-Lyndon word occurring in the $\omega$-Lyndon factorisation of $x[n]$ (see Proposition~\ref{prop1}). If $(l(n))_{n\geq 1}$ is unbounded, then pick $n$ such that i) $l(n)$ is greater than the length of the longest
$\omega$-Lyndon prefix of $x$ and ii) $l(n)$ is the length of the last $\omega$-Lyndon word in the $\omega$-Lyndon factorisation of $x[n].$
% I think we should justify this is possible
 Then  $x[n]=l_1l_2\cdots l_k$ with $l_i\in \Lyn$ and $l_1^\omega \succeq l_2^\omega \succeq \cdots \succeq  l_{k}^\omega,$ and $l_k$ is not a prefix of $x.$ By iteration of Lemma~\ref{lem11},  $(l_1l_2\cdots l_k)^\omega \succeq l_k^\omega$ and hence $(l_1l_2\cdots l_k)^\omega \succ l_k^\omega.$  Writing $x=l_1l_2\cdots l_{k-1}y$ with $y\in \A^\nats,$ since $l_k$ is a prefix of $y$ but not of  $l_1l_2\cdots l_k$ we have $x\succ y$ and hence $x\notin \Lyn.$  If $(l(n))_{n\geq 1}$ is bounded
then pick a finite set $F\subseteq \Lyn$ such that all $\omega$-Lyndon words occurring in the $\omega$-Lyndon factorisation of $x[n]$ for $n\in \nats$ belong to $F.$ Because of the non increasing order condition in the  $\omega$-Lyndon factorisation, there exist $l_1,l_2,\ldots ,l_k$ in $F$ with $l_1^\omega \succeq l_2^\omega \succeq \cdots \succeq l_k^\omega$ such that  $l_1l_2\cdots l_{k-1}l_k^m$ is a prefix of $x$ for every $m\in \nats.$ Pick $m$ such that $l_k^m$ is not a prefix of $x$ and write $x=l_1l_2\cdots l_{k-1}y$ with $y\in \A^\nats.$ Then as $l_k^m$ is a prefix of $y$ but not of $l_1l_2\cdots l_k^m$ and $(l_1l_2\cdots l_k^m)^\omega \succeq l_k^\omega$ we deduce that $x\succ y$ and hence $x\notin \Lyn.$ \end{proof}

\begin{proposition}\label{existence} Each  $x\in \A^\nats$ admits either an infinite or a finite $\omega$-Lyndon factorisation.\end{proposition}

\begin{proof} Let $x\in \A^\nats $ and assume $x$ does not admit an infinite $\omega$-Lyndon factorisation. We  will show that $x$ admits a finite $\omega$-Lyndon factorisation.  The result is immediate in case $x\in \Lyn$ so we may assume that $x\notin \Lyn.$ For $n\in \nats,$ let $l_i^{(n)}$ be the $i$'th $\omega$-Lyndon word occurring in the $\omega$-Lyndon factorisation of $x[n],$ where $x[n]$ is the prefix of length $n$ of $x.$ In other words $x[n]=l_1^{(n)}l_2^{(n)}\cdots l_k^{(n)}.$ We may take $l_i^{(n)}$ to be the empty word if the $\omega$-Lyndon factorisation of $x[n]$ has fewer than $i$ terms.  By Lemma~\ref{L1}, the set $L_1=\{l_1^{(n)}\,:\, n\in \nats\}$ is finite and hence there exist $l_1\in \Lyn$ and an infinite set $A_1\subseteq \nats$ such that  for each $n\in A_1$ the $\omega$-Lyndon factorisation of $x[n]$ begins in $l_1.$  Put $L_2= \{l_2^{(n)}\,:\, n\in A_1\}.$  If $L_2$ is finite, then we may pick  $l_2 \in \Lyn$ and an infinite subset $A_2\subseteq A_1$ such that for each $n\in A_2$ the $\omega$-Lyndon factorisation of $x[n]$ begins in $l_1l_2$ and put $L_3= \{l_3^{(n)}\,:\, n\in A_2\}.$ Continuing as above, if each $L_k$ is finite then $x$ would admit an infinite $\omega$-Lyndon factorisation contrary to our assumption. And hence, there exists $k\geq 2$  and $l_1,l_2,\ldots ,l_{k-1}\in \Lyn$ and an infinite set $A_{k-1}\subseteq \nats$ such that for each $n\in A_{k-1}$ the $\omega$-Lyndon factorisation of $x[n]$ begins in $l_1l_2\cdots l_{k-1}$ and $L_k= \{l_k^{(n)}\,:\, n\in A_{k-1}\}$ infinite. Define $l_k\in \A^\nats$ by $x= l_1l_2\ldots l_{k-1}l_k.$ We claim $l_k\in \Lyn$ and $l_{k-1}^\omega \succ l_k.$  Observe that $l_{k}\neq l_{k-1}^\omega$ for otherwise $x=l_1\cdots l_{k-2}l_{k-1}^\omega$ is an infinite $\omega$-Lyndon factorisation of $x.$ Pick $m\in \nats$ such that $l_{k-1}^m$ is not a prefix of $l_k$ and $n\in A_{k-1}$ such that $|l_{k-1}^m| <|l_k^{(n)}|.$  Since $l_{k-1}^\omega \succeq (l_k^{(n)})^\omega$ and $l_{k-1}^m$ is  not a prefix of $l_k^{(n)},$ it follows that  $l_{k-1}^\omega \succ l_k.$ It remains to show that $l_k\in \Lyn.$ By Lemma~\ref{L1}, if $l_k\notin \Lyn$ then
$l_k=u^\omega$ for some $u\in \Lyn$ and hence $x=l_1l_2\cdots l_{k-1}u^\omega$ is an infinite $\omega$-Lyndon factorisation, a contradiction. 
\end{proof} 

\noindent We now turn to the question of unicity of $\omega$-Lyndon factorisations for infinite words. We begin by establishing unicity for words admitting a finite $\omega$-Lyndon factorisation.

\begin{lemma}\label{un1} Let $x\in \A^\nats$ and $u_1u_2\cdots u_k$ be a prefix of $x$  such that  $u_1^\omega \succeq u_2^\omega \succeq  \cdots \succeq u_k^\omega.$  If $x\in \Lyn$  then each $u_i$ is a prefix of $x.$ 
\end{lemma}

\begin{proof} By iteration of Lemma~\ref{lem11},  for  $1\leq i\leq k$ we have that $(u_1\cdots u_i)^\omega \succeq u_i^\omega.$ Let $v_i$ denote the prefix of $x$ of length $|u_i|.$ If $v_i\neq u_i$ then $v_i^\omega \succ u_i^\omega$ contradicting that $x\in \Lyn.$
\end{proof}

\begin{lemma}\label{un2} Let $x\in \A^\nats$ and $k\geq 3.$ Assume $u_1u_2\cdots u_k$ is a prefix of $x$ such that  $u_1^\omega \succeq u_2^\omega \succeq  \cdots \succeq u_k^\omega.$ If $x\in \Lyn,$ then either $|u_1\cdots u_{k-2}|\leq |u_k|$ or $u_1\cdots u_{k-2}u_k$ is a prefix of $x.$
\end{lemma}

\begin{proof}Assume $|u_1\cdots u_{k-2}|>|u_k|.$ By Lemma~\ref{un1}, we have that $u_k$ is a prefix of $x$ and hence a prefix of $u_1\cdots u_{k-2}.$   By Lemma~\ref{lem11}  we have that $(u_1\cdots u_{k-2})^\omega \succeq u_{k-2}^\omega \succeq u_{k-1}^\omega$ and hence $(u_1\cdots u_{k-2}u_{k-1})^\omega \succeq(u_{k-1}u_1\cdots u_{k-2})^\omega.$ Since $x\in \Lyn$ it follows that $u_{k-1}u_k$ is a prefix of $x$ and hence $u_k$ is a prefix of $u_{k-1}u_k.$ Thus $u_1\cdots u_{k-2}u_k$ is a prefix of $x.$
\end{proof}

\begin{lemma}\label{un3} Let $(u_i)_{i\in \nats}$ be a sequence in $\A^+$ with $u_1^\omega \succ u_2^\omega \succ  \cdots .$ Then $\prod_{i=1}^\infty u_i\notin \Lyn. $ 
\end{lemma}
\begin{proof} Put  $x=u_1u_2\cdots $ and suppose to the contrary that $x\in \Lyn.$  We will show that $|u_k|< |u_1|$ for each $k\geq 3$
and hence the sequence $(u_i)_{i\in \nats}$ is ultimately constant, a contradiction. To see that $|u_k|< |u_1|$ for each $k\geq 3,$  suppose to the contrary that $|u_k|\geq |u_1|$ for some $k\geq 3.$ By iteration of Lemma~\ref{un2}, there exists $2\leq j\leq k-1$ such that
$u_1\cdots u_ju_k$ is a prefix of $x$ and $|u_1\cdots u_{j-1}|\leq |u_k|.$ By Lemma~\ref{un1} we have that $u_k$ is a prefix of $x$ and hence $u_1\cdots u_{j-1}$ is a prefix of $u_k.$ As $(u_1\cdots u_{j-1})^\omega \succ u_j^\omega$ we have that $(u_1\cdots u_{j-1}u_j)^\omega \succ (u_ju_1\cdots u_{j-1})^\omega$ and hence $u_1\cdots u_{j-1}u_j\neq u_ju_1\cdots u_{j-1}.$ Since $u_1\cdots u_{j-1}$ is a prefix of $u_k$ it follows that the suffix of $x$ beginning in $u_ju_k$ is smaller than $x$ contradicting that $x\in \Lyn.$
\end{proof}

\begin{lemma}\label{un4} Let $x\in \A^\nats.$ If $x=v_1v_2v_3\cdots$ with $v_1^\omega\succeq v_2^\omega \succeq \cdots$ then $x\notin \Lyn.$
\end{lemma}

\begin{proof}Assume to the contrary that $x\in \Lyn.$ Without loss of generality we may assume that each $v_i$ is primitive. We claim $(v_i)_{i\geq1}$ is ultimately periodic. In fact, if the sequence $(v_i)_{i\geq 1}$ is not ultimately periodic, then by concatenating together the consecutive terms of the sequence which are equal, we may write $x=u_1u_2\cdots$  with $u_1^\omega \succ u_2^\omega \succ \cdots$  in contradiction with Lemma~\ref{un3}. 
As $x\in \Lyn$ and hence not periodic,  write $x=v_1\cdots v_kv_{k+1}^\omega$ with $v_1^\omega \succeq \cdots \succeq v_k^\omega \succ v_{k+1}^\omega$ and pick $m$ such that $v_{k+1}^m$ is not a prefix of $x.$ As $(v_1\cdots v_kv_{k+1}^m)^\omega \succ v_{k+1}^\omega,$ it follows that the suffix $v_{k+1}^\omega\prec x$ contradicting that $x\in \Lyn.$ 
\end{proof}

%we deduce that the sequence $(u_i)_{i\geq 1}$ is ultimately periodic and hence so is the sequence $(v_i)_{i\geq 1}.$ As $x\in \Lyn$ and hence not periodic,  write $x=v_1\cdots v_kv_{k+1}^\omega$ with $v_1^\omega \succeq \cdots \succeq v_k^\omega \succ v_{k+1}^\omega$ and pick $m$ such that $v_{k+1}^m$ is not a prefix of $x.$ As $(v_1\cdots v_kv_{k+1}^m)^\omega \succ v_{k+1}^\omega,$ it follows that the suffix $v_{k+1}^\omega\prec x$ contradicting that $x\in \Lyn.$ 

\begin{lemma}\label{infin} If $x$ admits an infinite $\omega$-Lyndon factorisation, then no suffix of $x$ belongs to $\Lyn.$ In particular $x$ does not admit a finite $\omega$-Lyndon factorisation.  
\end{lemma}

\begin{proof} Suppose $x$ admits an infinite $\omega$-Lyndon factorisation $x=l_1l_2l_3\cdots.$ Then any suffix $y$ of $x$ may be written as $y=s_il_{i+1}l_{i+2}\cdots$ with $i\geq 1$ and $s_i$ a suffix of $l_i.$ Since $s_i^\omega \succeq l_i^\omega$ it follows that $s_i^\omega \succeq l_{i+1}^\omega \succeq \cdots.$ By Lemma~\ref{un4}, it follows that $y\notin \Lyn.$
\end{proof}

\begin{corollary}\label{finfact}If an infinite word $x\in \A^\nats$ admits a finite $\omega$-Lyndon factorisation $x=l_1l_2\cdots l_k,$ then it is the unique $\omega$-Lyndon factorisation of $x.$ 
\end{corollary}

\begin{proof}It follows from Lemma~\ref{infin} that $x$ does not admit an infinite $\omega$-Lyndon factorisation. It remains to show that $x$ admits no other finite $\omega$-Lyndon factorisation. For this purpose, write $x=ul_k$ with $u\in \A^*$ and observe that if $v\in \A^+$ is any suffix of $u,$  then (by iteration of Lemma~\ref{lem11})  $l_k\preceq vl_k.$ In other words, $vl_k\notin \Lyn$ and hence $l_k$ is necessarily the first $\omega$-Lyndon suffix of $x.$ Unicity now follows from Proposition~\ref{prop1}.
\end{proof}

We next prove unicity of $\omega$-Lyndon factorisations for those infinite words $x$ not admitting a finite $\omega$-Lyndon factorisation.
We first consider the case that $x$ is ultimately periodic:

\begin{lemma}\label{upcase} Assume $x\in \A^\nats$ is ultimately periodic. Then $x$ admits a unique $\omega$-Lyndon factorisation. 
\end{lemma}  

\begin{proof} By Corollary~\ref{finfact} we may suppose that $x$ does not admit a finite $\omega$-Lyndon factorisation. Using Proposition~\ref{existence} let $x=l_1l_2\cdots $ be an infinite $\omega$-Lyndon factorisation of $x.$ We claim the sequence $(l_i)_{i\in \nats}$ is ultimately constant.
It suffices to show that $\liminf_{i\rightarrow \infty}|l_i|<+\infty.$ So pick a suffix $x'$ of $x$ and an infinite set $I\subseteq \nats$ such that $l_i$ is a prefix of $x'$ for each $i\in I.$ Using lemmas~\ref{L1} and \ref{infin} it follows that either  $x'=l^\omega$ for some $l\in \Lyn$ or $x'$ has only finitely many $\omega$-Lyndon prefixes. In the latter case  $\{|l_i|\,:\, i\in I\}$ is clearly bounded. In the former case pick $j<k$ in $I$ such that $x'=\prod_{i\geq j}l_i$ and $\min\{|l_j|,|l_k|\}\geq 2|l|.$ Then as $\omega$-Lyndon words are primitive it follows that $w=l_j\cdots l_{k-1}=l^r$ for some $r\in \nats$ contradicting Proposition~\ref{prop1}.

Having proved that any infinite $\omega$-Lyndon factorisation of $x$ is ultimately contant, unicity of the factorisation now follows. In fact, suppose $x=l'_1l'_2\cdots$ is another $\omega$-Lyndon factorisation with $l'_i=l'$ for all $i$ greater than some $k'$ and $l'\in \Lyn.$ Then since $l$ and $l'$ are each primitive, it follows that $|l|=|l'|$ whence $l=l'$ and the two factorisations must ultimately synchronise, i.e., $l_i=l'_i$ for all sufficiently large $i.$  The rest now follows from Proposition~\ref{prop1}.  \end{proof}

% In fact $l$ is the smallest conjugate of a primitive period of a periodic suffix of $x.$

%Pick a periodic suffix $x'\in \A^\nats$ of $x$ and let $u$ be a primitive period of $x'.$ Let $l \in \Lyn$ be the smallest conjugate of $u,$ i.e., $l^\omega \preceq v^\omega$ for all conjugates $v$ of $u.$ We claim that there exists $k\in \nats$ such that  $l_i=l$ for  all $i\geq k.$   By Lemma~\ref{infin} no suffix of $x$ belongs to $\Lyn,$ and hence by the pigeonhole principle and Lemma~\ref{L1} it follows that $(|l_i|)_{i\in \nats}$ is bounded % (every factor is a prefix of a suffix)and hence the sequence $(l_i)_{i\in \nats}$ is ultimately constant, say equal to some $l'\in \Lyn.$ As $l$ and $l'$ are primitive it follows that $|l'|=|l|$ whence $l'=l.$ Having established the claim, since $l$ is primitive, it follows that any two infinite $\omega$-Lyndon factorisations necessarily synchronise. The rest now follows from Proposition~\ref{prop1}.  \end{proof}

\begin{definition} A factor $u\in \A^+$ of an infinite word $x$ is said to be minimal in $x$ if $u^\omega \preceq v^\omega$ for all factors $v$ of $x$ with $|v|=|u|.$ \end{definition}
\noindent We note that if $u$ is a minimal factor of $x$ then so is every prefix of $u.$ The following lemma will be applied to show that any infinite aperiodic word $x$ admits at most one infinite $\omega$-Lyndon factorisation, and how to construct it.

\begin{lemma}\label{M}
Assume $x\in \A^\nats$ and $u\in \A^+$ is a minimal factor of $x.$ Let $w\in \A^*$ be the longest prefix of $x$ preceding the first occurrence of $u$ in $x.$ 
Assume $x$ admits an infinite $\omega$-Lyndon  factorisation $x=l_1l_2l_3\cdots$ with $\limsup_{i\rightarrow \infty}|l_i|=+\infty.$ Then either $w=\varepsilon$ or $w=l_1\cdots l_k$ for some $k\in \nats.$
\end{lemma}

\begin{proof}
Put $n=|u|$ and write $u=u_1u_2\cdots u_n.$ Also write $x=wux'$ with  $x'\in \A^\nats;$  by assumption $wu$ contains exactly one occurrence of $u.$ Assume $w\neq \varepsilon$  and let $k$ be the least positive integer such that $\sum_{i=1}^k|l_i|\geq |w|.$ We must show that $\sum_{i=1}^k|l_i|= |w|.$ Suppose to the contrary that $\sum_{i=1}^k|l_i|> |w|.$ We first note that $u$ cannot be fully contained inside $l_k$ for otherwise, if $v$ denotes the prefix of $l_k$ with $|v|=|u|,$ then as $v\neq u$ and $l_k\in \Lyn$ we  have  $v^\omega\prec u^\omega$ which contradicts that $u$ is minimal. 
Thus we may write $l_k=zu_1\cdots u_p$ for some $z \neq \varepsilon$ and $p<n.$ Let $r=\min\{|l_i|\,:\,i\geq k+1\}$ and pick $j\geq k+1$ with $|l_j|=r.$ 
Also pick $j'>j$ such that $|l_{j'}|\geq p+r.$
%and write $l_j=x_1\cdots x_r.$ 

\textbf{Case 1: $n\geq p+r$}
\newline
By definition of $r$ it follows  that $u_{p+1}...u_{p+r}$ is a prefix of $l_{k+1}.$ We first claim that 
\begin{equation}u_{p+1}...u_{p+r} = u_1\cdots u_r = (u_1\cdots u_p)^{\frac{r}{p}} \end{equation}
where $(u_1\cdots u_p)^{\frac{r}{p}}$ denotes the prefix of length $r$ of $(u_1\cdots u_p)^\omega.$  In fact, as $u$ is a minimal factor of $x$ we have that $(u_1\cdots u_r)^\omega \preceq (u_{p+1}\cdots u_{p+r})^\omega.$ Furthermore since $l_{k+1}^\omega \preceq l_k^\omega \prec (u_1\cdots u_p)^\omega$ and $u_{p+1}\cdots u_{p+r}$ is a prefix of $l_{k+1}$
it follows that $(u_{p+1}\cdots u_{p+r})^\omega \preceq ((u_1\cdots u_p)^{\frac{r}{p}})^\omega.$ Combining we get $(u_1\cdots u_r)^\omega \preceq (u_{p+1}\cdots u_{p+r})^\omega \preceq ((u_1\cdots u_p)^{\frac{r}{p}})^\omega$ from which (1) follows.

\noindent We also claim that
\begin{equation}
l_j=u_1\cdots u_r.
\end{equation}
Indeed, since $u$ is a minimal factor of $x$ we have $(u_1\cdots u_r)^\omega \preceq l_j^\omega.$ On the other hand $l_j^\omega \preceq l_{k+1}^\omega$ and so by taking the prefix of length $r$ of both words we obtain $l_j^\omega \preceq (u_{p+1}\cdots u_{p+r})^\omega.$ So combining and using (1)  we deduce that $(u_1\cdots u_r)^\omega \preceq l_j^\omega\preceq (u_1\cdots u_r)^\omega$ from which (2) follows.

\noindent Thus we have \[(u_1\cdots u_r)^\omega =l_j^\omega \preceq l_k^\omega \prec (u_1\cdots u_p)^\omega\]  and hence by Lemma~\ref{lem11}

\begin{equation}
(u_1\cdots u_ru_1\cdots u_p)^\omega \prec (u_1\cdots u_pu_1\cdots u_r)^\omega.
\end{equation}

\noindent Using the fact $u_1\cdots u_pu_{p+1}\cdots u_{p+r}$ is a minimal factor of $x$ together with  (1) and (2) gives
\[(u_1\cdots u_pu_1\cdots u_r)^\omega =(u_1\cdots  u_{p+r})^\omega \preceq (l_{j'}[p+r])^\omega\preceq (l_j^\omega [p+r])^\omega=((u_1\cdots u_r)^{\frac{p+r}{r}})^\omega.\]
Together with (3) gives
\begin{equation}(u_1\cdots u_ru_1\cdots u_p)^\omega \prec (u_1\cdots u_pu_1\cdots u_r)^\omega \preceq ((u_1\cdots u_r)^{\frac{p+r}{r}})^\omega.\end{equation}
It follows from (4) that $u_1\cdots u_p=(u_1\cdots u_r)^{\frac pr}$ and hence 
\[u_1\cdots u_ru_1\cdots u_p=(u_1\cdots u_r)^{\frac{p+r}{r}}\] 
which by (4) gives $ ((u_1\cdots u_r)^{\frac{p+r}{r}})^\omega \prec  ((u_1\cdots u_r)^{\frac{p+r}{r}})^\omega,$ a contradiction.

\textbf{Case 2: $n<p+r$.}
\newline
In this case $u_{p+1}\cdots u_n$ is a prefix of $l_{k+1}$ and the same arguments used to prove (1) shows that 
\begin{equation}
u_1\cdots u_n = (u_1 \cdots u_p)^{\frac{n}{p}}.
\end{equation}
As $l_k^\omega =(zu_1 \cdots u_p)^\omega \prec (u_1 \cdots u_p)^\omega$, it follows that $|l_k|=|z|+p < n$ for otherwise  $u=u_1\cdots u_n$ would be a prefix of $l_k$ which would imply an earlier occurrence of $u$ in $x.$ Thus
\begin{equation}
zu_1\cdots u_p=(u_1 \cdots u_p)^{\frac{|z|+p}{p}}=(u_1 \cdots u_p)^au_1\cdots u_q
\end{equation}
for some choice of integers $a,q$ and as $l_k$ is primitive we have that $1\leq q <p.$ 

Finally, we have $z=(u_1 \cdots u_p)^{a-1}u_1\cdots u_q$ and hence
\begin{equation}
u_1 \cdots u_pu_1\cdots u_q=u_1 \cdots u_qu_1\cdots u_p
\end{equation}
from which it follows that $l_k$ is not primitive, a contradiction. 
\end{proof}

\begin{proposition}\label{uniinf} Let $x\in \A^\nats$ be an aperiodic infinite word and $x=l_1l_2l_3\cdots=l_1'l_2'l_3'\cdots $ two infinite $\omega$-Lyndon factorisations of $x.$ Then $l_i=l'_i$ for each $i\in \nats.$  
\end {proposition}

\begin{proof} Suppose to the contrary that $l_i\neq l'_i$ for some $i\in \nats.$ Short of replacing $x$ by some suffix of $x,$ we may assume that $l_1\neq l'_1.$ By Lemma~\ref{infin} it follows that $x\notin \Lyn$ and hence $x$ contains a minimal factor $u$ which is not a prefix of $x.$   Let $w\in \A^+$ denote the prefix of $x$ which precedes the first occurrence of $u$ in $x.$ As $x$ is aperiodic it follows that $\limsup_{i\rightarrow \infty}|l_i|= \limsup_{i\rightarrow \infty}|l'_i|=+\infty.$
By Lemma~\ref{M} it follows that there exist $k,k'\in \nats$ such that $w=l_1\ldots l_k=l'_1\cdots l'_{k'}$ contradicting Proposition~\ref{prop1}. \end{proof}

\begin{corollary}Each infinite word $x\in \A^\nats$ admits  precisely one $\omega$-Lyndon factorisation.
\end{corollary}  

\begin{proof} Existence follows from Proposition~\ref{existence}. For unicity, if $x$ admits a finite $\omega$-Lyndon factorisation, then unicity follows from Corollary~\ref{finfact}. So we may suppose that  $x$ admits only infinite $\omega$- Lyndon factorisations. If  $x$ is ultimately periodic unicity follows from  Lemma~\ref{upcase} while if $x$ is aperiodic unicity follows from Proposition~\ref{uniinf}. \end{proof}

%\section*{References}

\end{document}